\documentclass[sts]{imsart}

\RequirePackage{amsthm,amsmath,amsfonts,amssymb,booktabs}
\RequirePackage[authoryear]{natbib}
\RequirePackage[colorlinks,citecolor=blue,urlcolor=blue]{hyperref}

\startlocaldefs
\theoremstyle{plain}
\newtheorem{theorem}{Theorem}

\newtheorem{corollary}{Corollary}

\theoremstyle{definition}
\newtheorem{remark}{Remark}
\newtheorem{example}{Example}
\newtheorem{definition}{Definition}

\DeclareMathOperator{\var}{\text{var}}

\newcommand{\df}{df}
\DeclareMathOperator{\se}{\text{se}}

\DeclareMathOperator{\SEF}{\text{SEF}}
\DeclareMathOperator{\BF}{\text{BF}}
\newcommand{\tmax}{t_{\boldsymbol{R}^2}^{\max}}
\newcommand{\R}{\boldsymbol{R}}
\newcommand{\Ry}{R^2_{Y\sim Z | D\mathbf{X}}}
\newcommand{\Rd}{R^2_{D\sim Z | \mathbf{X}}}

\newcommand{\Rdlz}{R^2_{D\sim Z_L | \mathbf{X}}}
\newcommand{\Rymax}{R^{\max }_{Y}}
\newcommand{\Rdmax}{R^{\max }_{D}}

\newcommand{\Rdmv}{R^2_{D\sim \mathbf{Z} | \mathbf{X}}}

\newcommand{\fstar}{f^*_{\alpha, \df-1}}
\newcommand{\fstarsq}{f^{* 2}_{\alpha, \df-1}}

\newcommand{\tstar}{t^*_{\alpha, \df-1}}
\newcommand{\tstarsq}{t^{*2}_{\alpha, \df-1}}
\newcommand{\tobs}{t_{r}}

\newcommand{\RV}{\text{RV}}
\newcommand{\XRV}{\text{XRV}}
\newcommand{\XRVI}{\text{XRVI}}
\newcommand{\RVI}{\text{RVI}}

\newcommand{\Rystar}{R^{* 2}_{Y\sim Z | D\mathbf{X}}}
\newcommand{\Rdstar}{R^{* 2}_{D\sim Z | \mathbf{X}}}
\newcommand{\Ryx}{R^2_{Y \sim \mathbf{X}|D}}
\newcommand{\Rdx}{R^2_{D \sim \mathbf{X}}}

\endlocaldefs

\begin{document}

\begin{frontmatter}
\title{On the minimum strength of (unobserved) covariates to overturn an insignificant result}
\runtitle{On the minimum strength of association to overturn an insignificant result}

\begin{aug}
\author[A]{\fnms{Danielle}~\snm{Tsao}\ead[label=e1]{dltsao@uw.edu}},
\author[B]{\fnms{Ronan}~\snm{Perry}\ead[label=e2]{rflperry@uw.edu}}
\and
\author[C]{\fnms{Carlos}~\snm{Cinelli}\ead[label=e3]{cinelli@uw.edu}}
\address[A]{University of Washington, Seattle, USA\printead[presep={\ }]{e1}.}
\address[B]{University of Washington, Seattle, USA\printead[presep={\ }]{e2}.}
\address[C]{University of Washington, Seattle, USA\printead[presep={\ }]{e3}.}
\end{aug}

\begin{abstract}
We study conditions under which the addition of variables to a regression equation can turn a previously statistically insignificant result into a significant one. Specifically, we characterize the minimum strength of association required for these variables---both with the dependent and independent variables, or with the dependent variable alone---to elevate the observed \mbox{t-statistic} above a specified significance threshold. 
Interestingly, we show that it is considerably difficult to overturn a statistically insignificant result solely by reducing the standard error.
Instead, included variables  must also alter the point estimate to achieve such reversals in practice. 
Our results can be used to conduct sensitivity analyses against unobserved variables and to bound the maximum t-value one can obtain given different subsets of observed covariates, and may also offer algebraic explanations for  patterns of  reversals seen in empirical research, such as those documented by \cite{lenz2021achieving}.
\end{abstract}


\end{frontmatter}

\section{Introduction}

There are many legitimate reasons for including covariates in a regression equation. These reasons arise both in experimental and observational settings. Consider, for example, a randomized controlled trial (RCT) in which a researcher uses ordinary least squares (OLS) to estimate the average effect of a treatment on an outcome. Although randomization eliminates confounding by design, including pre-treatment covariates can still improve the precision of estimated effects. In observational studies, beyond precision gains, accounting for covariates is often necessary to block spurious pathways, or mitigate other sources of bias \citep{cinelli2022crash}. Such concerns not only motivate the inclusion of observed covariates in a regression equation, but also raise the question of how results might have differed had certain unmeasured variables been included.

Despite these valid motivations, applied researchers often fail to adequately justify their choice of control variables. And it is not uncommon for the statistical significance of main findings to depend on such choices. For example, in the American Journal of Political Science, \cite{lenz2021achieving} found that over 30\% of articles relied on the inclusion of covariates to turn previously statistically insignificant findings into significant ones.  According to \cite{lenz2021achieving}, none of the articles justified this choice, nor did they disclose these reversals. More generally, the practice of testing various model specifications until one obtains a statistically significant result---whether intentionally or not---appears to be prevalent across disciplines \citep{brodeur2016star,vivalt2019specification}.  These significance reversals raise several questions. Under what conditions can they occur? Are there natural limits to the t-statistic one can obtain via different choices of covariates? What observable patterns in the data should emerge when these reversals take place?

In this paper, we provide simple algebraic answers to these questions in the context of OLS. Building on recent results from \cite{cinelli_making_2020, cinelli2025omitted}, we first characterize the maximum change in the t-statistic that covariates with bounded strength can produce. We then derive the minimum strength of association that such covariates must have---whether with both the dependent and independent variables or with the dependent variable alone---to elevate the observed t-statistic above a given threshold. Lastly, we provide an empirical example. Our results can be used to conduct sensitivity analyses regarding unobserved ``suppressors'' and to bound the maximum attainable t-value resulting from the choice of observed control variables. They also offer algebraic explanations for patterns of significance reversals observed in empirical research.\footnote{Note that post-hoc selection of variables based on the data usually invalidates standard distributional results of test statistics---see, e.g., \cite{berk2013valid} for a proposal of valid post-selection inference in the context of linear regression. We clarify that such results are orthogonal to the main results of our paper. Here we answer the following question: \emph{when can the addition of variables to a regression equation elevate the observed t-statistic above a pre-specified threshold?} These results pertain to the algebra of least squares alone, and do not depend on assumptions on the underlying sampling distribution.}

\section{Preliminaries} \label{background}

\subsection{Problem setup}
Let $Y$ be an $(n \times 1)$ vector containing the dependent variable  for $n$ observations; $D$ be an $(n \times 1)$ vector of the independent variable of interest and $\mathbf{X}$ be an $(n \times p)$ matrix of observed covariates including a constant. Consider the regression equation
\begin{align}
    Y = \hat{\lambda}_{r }D + \mathbf{X}\hat{\beta}_r + \hat{\epsilon}_{r}, \label{resmodel}
\end{align}
where $\hat{\lambda}_{r}$, $\hat{\beta}_r$ are the OLS estimates of the regression coefficients of $Y$  on $D$ and $\mathbf{X}$, and $\hat{\epsilon}_{r}$ is the corresponding $(n \times 1)$ vector of residuals. 

Let $\widehat{\se}(\hat{\lambda}_r)$ be the estimated classical (homoskedastic) standard error of $\hat{\lambda}_r$. Under the classical linear regression model, the t-statistic for testing the null hypothesis \mbox{$H_0:\lambda_r = \lambda_0$}, i.e.,
\begin{align}
t_{r}:= \frac{\hat{\lambda}_r - \lambda_0}{\widehat{\se}(\hat{\lambda}_r)}, \label{t-res}
\end{align}
 follows a t-distribution with $\df:= n-p-1$ degrees of freedom. Denoting by $t^*_{\alpha, \df}$ the \mbox{$(1-\alpha/2)$} quantile of  this distribution, the t-statistic (\ref{t-res}) is considered ``statistically significant with significance level $\alpha$'' if the absolute value of $t_{r}$ exceeds that of $t^*_{\alpha, \df}$. 
 Note that the t-statistic depends on the choice of $\lambda_0$. For simplicity, we use the notation $t_r$ with the understanding that a particular $\lambda_0$ has been chosen.

Now suppose the t-statistic (\ref{t-res}) is insignificant. Let $Z$ be an $(n\times1)$ vector of a (potentially unobserved) covariate whose inclusion in the regression equation we wish to assess. In contrast to the ``restricted" regression in \eqref{resmodel}, we now consider the long regression equation of $Y$ on $D$ after adjusting for both $\mathbf{X}$ and $Z$,
\begin{align} 
    Y = \hat{\lambda}D + \mathbf{X}\hat{\beta} + \hat{\gamma}Z + \hat{\epsilon}.\label{regmodel}
\end{align}
Here, the t-statistic for testing \mbox{$H_0:\lambda = \lambda_0$} is
\begin{align}
t:= \frac{\hat{\lambda} - \lambda_0}{\widehat{\se}(\hat{\lambda})} \label{t-long}
\end{align}
where $\hat{\lambda}$ and $\widehat{\se}(\hat{\lambda})$ have the same interpretation as before, just now with an additional adjustment for $Z$. 
We wish to quantify the properties that $Z$ needs to have such that the t-statistic in (\ref{t-long}) will exceed the critical threshold $t^*_{\alpha, \df-1}$.

\begin{remark}
\label{rmk:multiple}
Throughout the paper we illustrate results using the critical threshold $t^*_{\alpha, \df-1}$, which is usually justified by classical linear model theory. However, we remind the reader that researchers may choose the critical threshold as they see fit, possibly justified by different distributional assumptions. For example, if a researcher is testing $k$ individual hypotheses, they may consider the Bonferroni-corrected threshold $t^*_{\alpha/k, \df-1}$.
\end{remark}

\subsection{Omitted variable bias formulas}

Comparing (\ref{t-res}) with (\ref{t-long}),
observe that the relative change in the absolute value of the
t-statistic can be decomposed as the product of the relative change 
in the bias
and the relative change in the standard error:
\begin{align}
\left|\frac{t}{t_r}\right|= \left|\frac{\hat{\lambda} - \lambda_0}{\hat{\lambda}_r-\lambda_0} \right| \times \left(\frac{{\widehat{\se}(\hat{\lambda}_r)}}{\widehat{\se}(\hat{\lambda})}\right). \label{eq:t.tr}
\end{align}
Concretely, for $Z$ to double the t-statistic, it must either double the absolute difference between the point estimate and $\lambda_0$, halve the standard errors, or achieve some combination of both.

To characterize these changes in terms of how much residual variation $Z$ explains of $D$ and $Y$, we refer to the following result from \cite{cinelli_making_2020}. In what follows, we define the sample partial $R^2$ of $A$ with $B$ given $C$ as  $R^2_{A\sim B|C} := 1 - \frac{\widehat{\var}(A^{\perp B,C})}{\widehat{\var}(A^{\perp C})}$, where $\widehat{\var}(\cdot)$ denotes the sample variance, and $A^{\perp B,C}$ denotes the OLS residual from regressing $A$ on $B$ and $C$.
\begin{theorem}[\citealp{cinelli_making_2020}]
\label{thm:ovb}
    Let $\Ry$ denote the sample partial $R^2$ of $Y$ with $Z$ after adjusting for $D$ and $\mathbf{X}$, and let $\Rd < 1$ denote the sample partial $R^2$ of $D$ with $Z$ after adjusting for $\mathbf{X}$. Then,
\begin{align}
    |\hat{\lambda}_{r } - \hat{\lambda}| = \underbrace{\sqrt{ \frac{\Ry\Rd}{1-\Rd}}}_{\textrm{BF}} \times \widehat{\se}(\hat{\lambda}_{r }) \times \sqrt{\df} \label{bf}
\end{align}
and 
\begin{align}
    \widehat{\se}(\hat{\lambda}) 
    = \underbrace{\sqrt{\frac{1-\Ry}{1-\Rd}}}_{\textrm{SEF}}  \times \widehat{\se}(\hat{\lambda}_{r }) \times\sqrt{\frac{\df}{\df-1}}. \label{selambda}
\end{align}
To aid interpretation, we call the terms BF in (\ref{bf}) and SEF in (\ref{selambda}) the ``bias factor'' and the ``standard error factor'' of $Z$, respectively. 
\label{ovb}
\end{theorem}

We can use Theorem~\ref{thm:ovb}  to write the absolute value of the t-statistic~(\ref{t-long}) as a function of $\Ry$ and $\Rd$, i.e.,
\begin{align}\label{eq:abs_tstat}
        {\textstyle t(\Ry, \Rd)= 
        \frac{|(\hat{\lambda}_{r } - \lambda_0) \pm \BF \times \widehat{\se}(\hat{\lambda}_{r }) \times \sqrt{\df}|}{\widehat{\se}(\hat{\lambda}_{r })\times\SEF\times\sqrt{\frac{\df}{\df-1}}},} 
\end{align}
where the sign of the bias term, denoted by $\pm$, depends on whether $\hat{\lambda}_r > \hat{\lambda}$ or vice-versa.
This re-formulation allows us to assess how $Z$ affects inferences for any postulated pair of partial $R^2$ values $\{\Ry, \Rd\}$, and it will help us determine the conditions under which the addition of $Z$ turns a previously statistically insignificant result into a significant one.

To set the stage for upcoming results, we note an immediate but important corollary of Theorem~\ref{thm:ovb}:
for a fixed observed t-statistic and fixed partial $R^2$ values of $Z$, the impact that $Z$ has on the relative bias depends on the sample size, whereas the impact it has on the relative change in standard errors does not. 
That is, the relative change in the bias is given by
$$
\frac{\hat{\lambda} - \lambda_0}{\hat{\lambda}_r-\lambda_0} = 1 \pm \frac{\BF}{t_r} \times \sqrt{\df}.
$$
For \emph{fixed $t_r$} and fixed $\{\Ry, \Rd\}$ (which sets $\BF$), 
larger sample sizes yield larger relative changes in the distance of the estimate from the null hypothesis.  Whereas, the relative change in the standard error is given by
\begin{align}
\frac{\widehat{\se}(\hat{\lambda}_r)}{\widehat{\se}(\hat{\lambda})}= \frac{1}{\SEF}\times \sqrt{\frac{\df-1}{\df}} \approx \frac{1}{\SEF}
\end{align}
and is thus unaffected by the sample size. As an example, 
halving standard errors is equally challenging in a sample of 100 as in a sample of 1,000,000; in contrast, \emph{given a fixed $t_r$}, doubling the point estimate becomes much easier as the sample size grows. This distinction will become clearer as we continue with our analysis.

\section{Results}

In this section we present two main results.
First, given upper-bounds on $\Ry$ and $\Rd$, we derive the ``maximum adjusted t-statistic'' which quantifies the maximum possible value that $t(\Ry, \Rd)$ can attain after including $Z$ in the regression equation. Then we solve for the minimum upper bound on $\{\Ry, \Rd\}$, hereby referred to as the ``strength of $Z$,''  such that it guarantees that the maximum adjusted t-statistic exceeds the desired threshold. 

In what follows, it is useful to define the quantities
\begin{align*}
        f_{r} := |t_{r}| / \sqrt{\df} \quad \text{and} \quad f^*_{\alpha, \df} := t^*_{\alpha, \df} / \sqrt{\df},
\end{align*}
which normalize the observed t-statistic and the critical threshold by the degrees of freedom. These definitions greatly simplify formulas and derivations. 

\subsection{On the maximum adjusted {t}-statistic}

We start by defining the maximum value that the t-statistic~(\ref{t-long}) could attain given $Z$ with  bounded strength.
    
\begin{definition}[Maximum adjusted t-statistic] For a fixed null hypothesis $H_0:\lambda = \lambda_0$, and  upper bounds on $\Ry$ and $\Rd$, denoted by $\R^2 = \{\Rymax,\Rdmax\}$, we define the maximum adjusted t-statistic as
\begin{align*}
    \tmax &:= \max_{\substack{\Ry, \Rd}} t(\Ry, \Rd)\\
    &\text{s.t.} \enspace \Ry \leq \Rymax, \Rd \leq \Rdmax.
\end{align*}
\end{definition}

 The solution to the above problem has a simple closed-form characterization. 

\begin{theorem}[Closed-form solution to $t_{\boldsymbol{R^2}}^{\max}$]
\label{thm:t-max}
Let $\Rymax < 1$. Then,
$$
{ \textstyle t_{\boldsymbol{R^2}}^{\max} =  \frac{f_{r} \sqrt{1-\Rdstar} + \sqrt{\Rystar\Rdstar}}{\sqrt{(1-\Rystar) / (\df-1)}} }
$$
where 
$$
{ \textstyle  \{\Rystar, \Rdstar\} =  \{\Rymax, \Rdmax \} }
$$
if $f_{r} ^2 < \Rymax(1-\Rdmax)/\Rdmax$ and 
$$
{ \textstyle  \{\Rystar, \Rdstar\} =  \left\{\Rymax, \frac{\Rymax}{f_{r} ^2 + \Rymax}\right\} }
$$
otherwise.
\end{theorem}

If $\tmax < \tstar$, then we can be assured that no $Z$ with the specified maximum strength would be able to overturn an insignificant result. On the other hand, if $\tmax > \tstar$, we know that there exists at least one $Z$ with strength no greater than $\R^2$ that is capable of bringing the t-statistic above the specified threshold. We note that the sign of the bias term in definition \eqref{eq:abs_tstat} of $t(\Ry, \Rd)$ is resolved by pushing $t_r$ further in the direction of $\text{sign}(t_r)$. 

\begin{remark} 
\label{rmk:r2y}
It is always necessary to constrain the strength of $Z$ with respect to $Y$ in order to obtain a finite solution for $\tmax$. If $\Rymax = 1$, then as $\Ry$ approaches one, the standard error approaches zero and the t-statistic grows without bounds. 
\end{remark}

\begin{remark}
\label{rmk:r2d}
In contrast,  it is possible to leave $\Rd$ unconstrained. While the optimal value of $\Ry$ always reaches the upper bound $\Rymax$, $\Rd$ may either reach its upper bound $\Rdmax$ or result in the interior point solution $\frac{\Rymax}{f_{r} ^2 + \Rymax}$.   This happens because increasing $\Rd$ has two counterbalancing effects on the t-statistic. On one hand, it changes the point estimate, as described by~(\ref{bf}), in a direction that is favorable for rejecting $H_0:\lambda = \lambda_0$.  On the other hand, it increases standard errors due to the variance inflation factor $\frac{1}{1-\Rd}$ in~(\ref{selambda}).  Intuitively, the more variation $Z$ explains of $D$, the less residual variation of $D$ is used to estimate $\hat{\lambda}$, which becomes noisier. Eventually, the increase in standard error is so large that it exceeds the benefit of the increase in point estimate, thus reducing the t-statistic. 
\end{remark}

Of course, there naturally could be multiple latent variables instead of a single one, and so one might wonder about the case when $Z$ takes the form of a matrix rather than a vector. The following theorem demonstrates that it is sufficient to consider a single unmeasured latent variable, up to a correction in the degrees of freedom.

\begin{theorem}[$t_{\boldsymbol{R^2}}^{\max}$ for matrix $\mathbf{Z}$]
\label{thm:multi-tmax}
Let $\mathbf{Z}$ be an $(n~\times~m)$ matrix of covariates. Then the solution to $\tmax$ is the same as that of Theorem~\ref{thm:t-max}, save for the adjustment in the degrees of freedom, which now is $\df-m$. 
\end{theorem}

In what follows, for simplicity we keep $Z$ as a vector with the understanding that all results still hold for matrix $\mathbf{Z}$. 
But before moving forward, there is an interesting corollary of the previous result. 
It characterizes the maximum extent to which one can p-hack results by testing different subsets of a fixed set of observed covariates $\mathbf{X}$.

\begin{corollary}[Upper bound on the t-statistic for any subset of covariates] \label{cor-phacking} For observed covariates $\mathbf{X}$, let $\Ryx, \Rdx$ denote
the strengths of the associations of  $\mathbf{X}$ with $Y$ and $D$, respectively. Let $t_{Y\sim D|\mathbf{X}_{S}}$ denote the t-statistic for the coefficient of the regression of $Y$ on $D$ when adjusting for a subset of covariates $\mathbf{X}_{S}$ (consisting of a subset of the columns of $\mathbf{X}$). Then for any $\mathbf{X}_{S}$,
$$
|t_{Y\sim D|\mathbf{X}_{S}}| \leq \tmax,
$$ 
where $\tmax$ is the solution of Theorem~\ref{thm:t-max} using $t_r = t_{Y\sim D}$ and $\R^2 = \{R^2_{Y \sim \mathbf{X}|D}, R^2_{D \sim \mathbf{X}}\}$.
\end{corollary}

When the number of covariates is small, it is feasible to run all possible regressions to identify the exact maximum t-statistic across all specifications. However, when the number of covariates is large, this exhaustive approach becomes impractical. For example, with $p = 40$, there are $2^{40}$ (approximately 1 trillion) possible specifications. {In such a case, $\tmax$ offers a simple upper bound on the maximum t-value that one could obtain via specification searches, without the need to actually run all 1 trillion regressions.}

\begin{example}
    Let $p=40$, $\df=100$ and the t-statistic of the regression of $Y$ on $D$ be equal to $1$.  Then, if $\{R^2_{Y \sim \mathbf{X}|D}, R^2_{D \sim \mathbf{X}}\} = \{0.08, 0.08\}$, Corollary~\ref{cor-phacking} assures us that none of the 1 trillion specifications can yield a t-statistic greater than $1.83$. 
\end{example}

\begin{remark}
    Note that the set of covariates $\mathbf{X}$ is arbitrary and can  thus accommodate any nonlinear transformations of the original variables. For example, if we again consider an original set of $p=40$ covariates, but also logarithms, squares, and pairwise interactions, Corollary~\ref{cor-phacking} can be applied to characterize the maximum adjusted t-statistic of the $2^{3p + p(p-1)/2}$ possible specifications, including these nonlinear transformations.
\end{remark}

\begin{remark}
    Note that $\tmax$ is achievable when the only constraint on the variables to be included is their maximum explanatory power. For any given set of observed covariates, $\tmax$ is a potentially loose upper bound. It is possible to tighten this bound by applying the corollary iteratively within subsets of regressions.
    Obtaining tight bounds without running all of the $2^p$ possible regressions remains an open problem.
\end{remark}

\subsection{On the minimal strength of \texorpdfstring{$Z$}{Z} to reverse statistical insignificance}

Equipped with the notion of the maximum adjusted \mbox{t-statistic}, we can now characterize the minimum strength of $Z$ to obtain a statistically significant result. Following  the convention in \cite{cinelli_making_2020, cinelli2025omitted} we call our metrics ``robustness values for insignificance.'' They quantify how ``robust'' an insignificant result is to the inclusion of covariates in the regression equation.\footnote{Our metrics complement the ``robustness value'' (RV) and ``extreme robustness value'' (XRV) introduced in \cite{cinelli_making_2020, cinelli2025omitted}, which deal with the opposite problem  and quantify the minimum strength of $Z$ to turn a \emph{significant result insignificant}. See the supplementary material for further discussion.}

\subsubsection{Extreme robustness value for insignificance.}
As highlighted in Remark~\ref{rmk:r2y}, the parameter $\Ry$ is essential for assessing the potential of $Z$ to bring about a significant result, as it always needs to be bounded. Thus we begin by characterizing the minimal strength of association of $Z$ with $Y$ \emph{alone} in order to achieve significance.

\begin{definition}[Extreme Robustness Value for Insignificance]
    For fixed $\Rdmax \in [0,1]$, the extreme robustness value for insignificance, $\XRVI^{\Rdmax}_{\alpha}$, is the minimum upper bound on $\Ry$ such that 
    $\tmax$ is large enough to reject null hypothesis $H_0: \lambda = \lambda_0$ at specified significance level $\alpha$, i.e.,
    \begin{align*}
        \XRVI^{\Rdmax}_{\alpha} &:= \min \{ \text{XRVI}:  t_{\text{XRVI},\Rdmax}^{\max} \geq \tstar \}.
    \end{align*}
\end{definition}

For a fixed bound on $\Rd$, the $\XRVI^{\Rdmax}_{\alpha}$ describes how robust an insignificant result is in terms of the minimum explanatory power that $Z$ needs to have with $Y$ in order to overturn it.
Theorem~\ref{thm:xriv-r} in the appendix provides an analytical expression for $\XRVI^{\Rdmax}_{\alpha}$ given arbitrary $\Rdmax\in[0,1]$. Here we focus on two important cases: $\Rdmax=0$ and $\Rdmax=1$.

Starting with $\Rdmax=0$, we first consider the scenario where the point estimate remains unchanged, and any increase in the t-statistic occurs solely due to a reduction in the standard error. That is, $\XRVI^{0}_{\alpha}$ quantifies how much variation a control variable  $Z$ that is uncorrelated with $D$ must explain of the dependent variable $Y$ in order to overturn a previously insignificant result.  This turns out to have  a remarkably simple and insightful characterization.

\begin{theorem}[Closed-form expression for $\XRVI^{0}_{\alpha}$]
\label{thm:XRVI-0}
    Let $f_{r} >0$, then the analytical solution for $\XRVI^{0}_{\alpha}$ is
    \begin{align*}
        \XRVI^{0}_{\alpha}
        &= 
        \begin{cases}
            0, &\text{if } \fstar \leq f_{r},  \\
            \displaystyle 1-\left(\frac{f_{r} }{\fstar}\right)^2,
            &\text{otherwise.} \\
        \end{cases}
    \end{align*}
    If $f_{r} = 0$, there is no value of $\Ry$ capable of overturning an insignificant result.
\end{theorem}

\begin{remark} 
\label{rmk:xrv0}
It is useful to understand how $\XRVI^{0}_{\alpha}$ changes as the sample size grows, when keeping the observed t-statistic and the significance level fixed.

\begin{align*}
    \XRVI^{0}_{\alpha} 
    \approx 1- \left(\frac{\tobs}{\tstar} \right)^2  
    \xrightarrow[\enskip\infty\enskip]{\df} \enspace 1- \left(\frac{\tobs}{z^*_{\alpha}} \right)^2,
\end{align*}
where here $ \xrightarrow[\enskip\infty\enskip]{\df}$ denotes the limit as $\df$ goes to infinity and $z^*_{\alpha}$ denotes the $(1-\alpha/2)$ quantile of the standard normal distribution. In other words, when considering only a reduction in the standard error, the amount of residual variation that $Z$ must explain of $Y$ in order to overturn an insignificant result depends solely on the ratio of the observed t-statistic to the critical threshold. Apart from changes in the critical threshold due to degrees of freedom---which eventually converges to $z^*_{\alpha}$---this value remains constant regardless of sample size.
\end{remark}

\begin{example}
     Consider testing the null hypothesis of zero effect with an observed t-statistic of $1$ and $100$ degrees of freedom.  The percentage of residual variation of $Y$ that $Z$ needs to explain in order to bring a t-statistic of 1 to the critical threshold of 2, \emph{only through a reduction in the standard error}, is
    $$
    \XRVI^{0}_{\alpha} \approx 1 - \left(\frac{1}{2}\right)^2 = 1 - (1/4) = 3/4 = 75\%.
    $$
    That is, if we are considering a reduction in the standard error alone, $Z$ needs to explain at least 75\% of the variation of $Y$ in order to elevate the observed t-statistic to~2. Notably, this number 
    is (virtually) the same across sample sizes, be it $\df = 100$, $\df = 1,000$ or $\df=1,000,000$. As variables that explain 75\% of the variation of $Y$ are rare in most settings, this simple fact suggests that it should also be rare to see a reversal of significance driven by gains in precision when $t_r =1$. We return to this point in the discussion.
\end{example}

The previous result describes how to achieve statistical significance via precision gains. We now move to the case with $\Rdmax=1$.  As noted in Remark~\ref{rmk:r2d}, this is the scenario in which we \emph{impose no constraints} on the strength of association between $Z$ and $D$. In this sense, $\XRVI^{1}_{\alpha}$ computes the \emph{bare minimum} amount of variation that $Z$ needs to explain of $Y$ in order to reverse an insignificant result. Any variable that does not explain at least $(100 \times \XRVI^{1}_{\alpha}) \%$ of the variation of $Y$ is logically incapable of making the t-statistic significant. 

\begin{theorem}[Closed-form expression for $\XRVI^{1}_{\alpha}$]
\label{thm:XRVI-1}
    The analytical solution for $\XRVI^{1}_{\alpha}$ is:
    \begin{align*}
        \XRVI^{1}_{\alpha}
        &= 
        \begin{cases}
            0, &\text{if } \fstar \leq f_{r},  \\
            \dfrac{\fstarsq - f_{r} ^2}{1+\fstarsq}, &\text{otherwise.} \\
        \end{cases}
    \end{align*}
\end{theorem}

\begin{remark} 
\label{rmk:xrv1}
Contrary to the previous case, we observe the following behaviour as the sample size grows,
\begin{align*}
    \XRVI^{1}_{\alpha} \approx  \left(\dfrac{\tstarsq-\tobs^2}{\df + \tstarsq}\right)  \xrightarrow[\enskip\infty\enskip]{\df} \enspace 0.
\end{align*}
Therefore, if we allow $Z$ to change point estimates, then for a fixed observed t-statistic, the minimal strength of $Z$ with $Y$ to bring about a reversal tends to zero as the sample size grows to infinity. 
\end{remark}

\begin{example}
\label{ex:tof1}
    Consider again testing the null hypothesis of zero effect with an observed t-statistic of $1$ and $100$ degrees of freedom. If we allow $Z$ to be arbitrarily associated with $D$, it needs only to explain 2.9\% of the residual variation of $Y$ in order to bring the t-statistic to 2:
    $$
    \XRVI^{1}_{\alpha} \approx \frac{2^2-1^2}{100 + 2^2} = \frac{3}{104} = 2.9\%.
    $$
    Also note that any $Z$ that explains less than 2.9\% of the variation of $Y$ is logically incapable of bringing about  such change. Corroborating our previous analysis, the $Z$ that achieves this must do  so via an increase in point estimate, and not via a decrease in standard errors. 
    As per Theorem~\ref{thm:t-max}, the optimal value of the association with $D$ is $\Rd\approx 74\%$. Notice that $\SEF \approx 1.94$, meaning that the inclusion of $Z$ almost \emph{doubles} the standard error, instead of reducing it. This, however, is compensated by the fact that $Z$ increases the point estimate by a factor of $1+\BF \times \sqrt{\df} = 3.88$, thus doubling the t-statistic despite the loss in precision.
        
\end{example}
\begin{example}
\label{ex:tof1df1000}
    For the same observed t-statistic of~1, consider a sample size that is an order of magnitude larger, say, $\df = 1,000$.  The minimum residual variation that $Z$ needs to explain of $Y$ then reduces to 0.29\%:
    $$
    \XRVI^{1}_{\alpha} \approx \frac{2^2-1^2}{1000 + 2^2} = \frac{3}{1004} = 0.29\%.
    $$
    As per Theorem~\ref{thm:t-max}, this $Z$ has an association with $D$ of $\Rd\approx 75\%$. Here, we have the same situation as before: the standard error doubles while the point estimate increases by a factor of four, thus doubling the t-statistic. 
    
\end{example}

\subsubsection{Robustness value for insignificance.} While in the previous section we investigated the minimal bound on $\Ry$ alone in order to reverse an insignificant result, here we investigate the minimal bound on both $\Ry$ and $\Rd$ simultaneously.

\begin{definition}[Robustness Value for Insignificance]
The robustness value for insignificance, $\RVI_{\alpha}$, is the minimum upper bound on both $\Ry$ and $\Rd$ such that 
$\tmax$ is large enough to reject the null hypothesis $H_0: \lambda = \lambda_0$ at specified significance level $\alpha$. That is,
    \begin{align*}
        \RVI_{\alpha} &:= \min \{ \RVI:  t_{\RVI,\RVI}^{\max} \geq \tstar \}. 
    \end{align*}
\end{definition}

Note that $\RVI_{\alpha}$ provides a convenient summary of the minimum strength of association that $Z$ needs to have, \emph{jointly} with $D$ and $Y$, in order to bring about a statistically significant result. Any $Z$ that has both partial $R^2$ values no stronger than $\RVI_{\alpha}$ cannot reverse a statistically insignificant finding. On the other hand,  we can always find a $Z$ with both partial $R^2$ values at least as strong as $\RVI_{\alpha}$ that does so. The solution of this problem is given in the following result.

\begin{theorem} [Closed-form expression for $\RVI_{\alpha}$]
\label{thm:RVI}
    The analytical solution for $\RVI_{\alpha}$ is
    \begin{align*}
        \RVI_{\alpha}
        &= 
        \begin{cases}
            0, &\text{if } \fstar \leq f_{r},  \\
            \tfrac{1}{2}\left({\scriptstyle  \sqrt{f_\Delta^4 + 4f_\Delta^2}-f_\Delta^2 }\right), &\text{if } f_{r} < \fstar < f_{r}^{-1}, \\
            \XRVI^{1}_{\alpha}, &\text{otherwise} \\
        \end{cases}
    \end{align*} 
    where $f_{\Delta} := \fstar - f_{r} $. 
\end{theorem}

\begin{remark}
\label{rmk:rvi-explanation}

The first case in Theorem~\ref{thm:RVI} occurs when the t-statistic for $H_0:\lambda = \lambda_0$ is already statistically significant, even when losing one degree of freedom. The second case occurs when both constraints on $\Rd$ and $\Ry$ are binding. The third case is the interior point solution, as defined in Theorem~\ref{thm:XRVI-1}, where only the constraint on $\Ry$ is binding. 
\end{remark}

Notice that in the second solution of \( \text{RVI}_{\alpha} \), we have \(\Rd = \Ry\); thus, \(\SEF = 1\) and standard errors remain unchanged. Therefore, \( \text{RVI}_{\alpha} \) represents the minimal strength of \( Z \) needed to achieve statistical significance via a \emph{change in the point estimate alone}, usefully complementing \(\XRVI^0_{\alpha}\), which quantifies the minimal strength of \( Z \) needed solely through a reduction in standard errors.

\begin{remark}
We recover $\XRVI^1_{\alpha}$ as the solution to $\RVI_{\alpha}$  if and only if the conditions (1) $\fstar > f_r^{-1}$ and (2) $\fstar > f_r$ both hold. This rarely occurs. To see this more clearly, note that condition
(1) simplifies to $\sqrt{\df-1}\sqrt{\df} < t_r\times {\tstar}$,  
or, approximately,
$$
\df \lessapprox t_r \times \tstar
$$
which only occurs when there are few degrees of freedom, for typical critical thresholds (e.g. 1.96).
\end{remark}

\begin{remark}
\label{rmk:rvi-df}
As with $\XRVI^1_{\alpha}$, for fixed $t_r$ and significance level $\alpha$, we observe the same behaviour for $\RVI_{\alpha}$ as the sample size grows,
    \begin{align*}
    \RVI_{\alpha}  \approx \frac{1}{2}\left(\sqrt{\frac{t_\Delta^4}{\df^2} + 4\frac{t_\Delta^2}{\df}} - \frac{t_\Delta^2}{\df}\right)  \xrightarrow[\enskip\infty\enskip]{\df} \enspace 0,
\end{align*}
where $t_\Delta = \tstar - \tobs$. Therefore, the larger the sample size, any change in the point estimate will eventually be sufficiently strong to bring about statistical significance.
\end{remark}

\begin{remark}
\label{rmk:order}
The statistics we introduced here obey the following ordering,
    \begin{align*}
    \XRVI^{1}_{\alpha} \leq \RVI_{\alpha}  \leq \XRVI^{0}_{\alpha}. 
\end{align*}
This follows directly from their definitions, as each case represents a constrained minimization problem and the constraint  
becomes stricter as we move from 
$\Rdmax = 1$ to $\Rdmax = 0$. Moreover, 
\mbox{$
\XRVI^{\RVI_\alpha}_{\alpha} = \RVI_{\alpha}
$}. 
\end{remark}

\begin{example}
Continuing with the case where the t-statistic is 1, we obtain $\text{RVI}_{\alpha} \approx 9.5\%$ when $\df=100$ and $\text{RVI}_{\alpha} \approx 3\%$ when $\df=1,000$. In both cases, the inflation of the t-statistic by a $Z$ that attains the optimal strength is driven solely by changes in the point estimate. 
\end{example}

In the supplementary material we discuss the sampling distribution of these statistics, both under the null and the alternative hypothesis.

\section{Empirical Example}\label{sec:empirical_example}
We demonstrate the use of our metrics in an empirical example that estimates the effect of a vote-by-mail policy on various outcomes \citep{amlani2022impact}. This work includes an analysis of the effect of a US county's vote-by-mail (VBM) policy  on the Republican presidential vote share (dependent variable $Y$) in the 2020 election. There are 5 treatment conditions concerning the VBM policy change from 2016 to 2020, of which the authors are specifically interested in the condition: no-excuse-needed (in 2016) to ballots-sent-in (in 2020), which we will refer to as \emph{condition-1}. The authors fit a differences-in-differences model using OLS, adjusting for various covariates including an indicator for battleground states and the median age and the median income of residents in the county. Notably, the interaction term for \emph{condition-1} $\times$ \emph{year} (independent variable $D$) is not statistically significant at the 5\% level: the t-statistic is $t_{r}=0.12$ with $4,307$ degrees of freedom. 

\begin{table}[t]
    \centering
    \begin{tabular}{rrrrrr}\toprule 
        Estimate & Std. Error & t-statistic & ${\XRVI^{1}_{\alpha}}$  &${\RVI_{\alpha}}$  &${\XRVI^{0}_{\alpha}}$ \\\midrule
        0.103 & 0.873 & 0.118 & $0.089\%$  & $2.77\%$  & $99.6\%$ \\ \bottomrule 
    \end{tabular}
    \smallskip
    {\raggedright\textbf{Note:} $\df = 4307$, $\lambda_0 = 0$, $\alpha = 0.05$. \hfill }
    \caption{Robustness values for insignificance for the vote-by-mail policy study.}
    \label{tab:amlani}
\end{table}

\subsection{Robustness to unobserved suppressors} The authors were concerned that the lack of significance for the coefficient of interest could have been due to suppression effects of unobserved variables. To address this, they use the formulas of Theorem~\ref{thm:ovb} to examine whether different hypothetical values for the strength of $Z$ yield a statistically significant t-statistic. Here we complement their analysis by providing the three proposed metrics, $\XRVI^1_\alpha$, $\RVI_\alpha$ and $\XRVI^0_\alpha$. 

The results are displayed in Table~\ref{tab:amlani}. We find that  any latent variable $Z$ that explains less than 2.77\% of the residual variation of both $Y$ and $D$ ($\RVI_{\alpha}=2.77\%$) would not be sufficiently strong to make the estimate statistically significant. Moreover, if we impose no constraints on $\Rd$, then $Z$  needs to explain at least 0.089\% of the residual variation of $Y$ in order to attain such a reversal ($\XRVI^1_{\alpha}=0.089\%$) . Finally, our analysis shows that a reversal of significance solely due to gains in precision is virtually impossible: a $Z$ orthogonal to $D$ would need to explain a remarkable $99.6$\% of the residual variation of $Y$   in order to overturn the insignificant result ($\XRVI^0_{\alpha}=99.6\%$).  

To put these statistics in context, a latent variable with the same strength of association with $D$ and $Y$ as that of the battleground state indicator would only explain 0.6\% of the residual variation in $Y$ and 0.027\% of the residual variation in $D$. Since both of these numbers are below the $\RVI_{\alpha}$ value of 2.77\%, we can immediately conclude that adjusting for a latent variable $Z$ of similar strength to this observed covariate would not be sufficient to overturn the insignificant result. 

\subsection{Robustness to subsets of controls}

We now illustrate how \(\tmax\) can be used to understand whether one can easily rule out the possibility of obtaining a statistically significant t-statistic when adjusting for different subsets of control variables. The authors present three model specifications, all of which found the effect of interest to be statistically insignificant. However, could there be a specification where the results turn out to be significant? Here we consider all possible variations between their base model and an expanded model that includes 12 additional control variables. This amounts to \(2^{12} = 4{,}096\) possible regressions. Applying the results of Corollary~\ref{cor-phacking}, we obtain \(\tmax \approx 20.7\), 
meaning that we cannot rule out that there exists a specification  where the interaction term becomes significant. Given the relatively small number of combinations, we can actually compute the ground truth to verify---and, indeed there are 510 models  that yield a statistically significant result. As per Remark~\ref{rmk:multiple}, one could also consider a similar exercise using a critical threshold that accounts for multiple testing.

\section{Discussion}

The algebra of OLS both imposes strong limits on and reveals clear patterns in how reversals of statistical insignificance occur. The first lesson that emerges from our analysis is that reversals of low t-statistics are unlikely to occur through reductions in standard errors alone. As shown in the first row of Table~\ref{tab:XRVI}, even with a t-statistic of 1.75, one would need to explain at least 20\% of the residual variation in \(Y\) to achieve statistical significance at the 5\% level solely through a reduction in standard errors. Elevating a t-statistic of .5 to statistical significance requires explaining a remarkable 93\% of the residual variation~of~$Y$.  Such strong associations with the response variable are not typically common in many empirical applications.

\begin{table}[b]
\centering
\begin{tabular}{lrrrrrrr}
\toprule
$t_r$ & 0.25 & 0.50 & 0.75 &1.00 & 1.25 & 1.50 & 1.75 \\
\midrule
$\XRVI^{0}_{\alpha=0.05}$ & 0.98 & 0.93 &0.85 & 0.74 & 0.59 & 0.41 & 0.20 \\[0.5em]
$\XRVI^{q_{.95}}_{\alpha=0.05}$ & 0.41 & 0.32 & 0.24 & 0.16 & 0.10 & 0.05 & 0.01\\
\bottomrule
\end{tabular}
{\raggedright\textbf{Note:} $\df = 1000$, $\tstar \approx 1.96$. \hfill }
\caption{Approximate values of XRVI for various values of $t_r$.}
\label{tab:XRVI}
\end{table}

A second consequence of our findings is that, in RCTs, it should be difficult to observe reversals of low t-statistics due to covariate adjustments. Since covariates in such trials typically have zero association with the treatment by design (barring sampling errors), their inclusion is unlikely to significantly shift the point estimate. A back-of-the-envelope calculation illustrates this point: if \( D \) is randomized, then \(\df \Rd\) asymptotically follows a chi-squared distribution with one degree of freedom (see supplement). Thus, letting \( q_{.95} \) denote the asymptotic approximation of the 95th percentile for \(\Rd\), we can calculate \(\XRVI^{q_{.95}}_{0.05}\) for various values of $t_r$. These values are recorded in the second line of Table~\ref{tab:XRVI}. With the exception of \( t_r = 1.75 \) and \( t_r = 1.5 \), the values of \(\Ry\) required to reverse statistical insignificance remain moderate to large, suggesting that such reversals should be uncommon in practice. 

Finally, and perhaps counter-intuitively, even when included variables are highly predictive of the response, reversals of insignificance are still typically  driven by shifts in the point estimate rather than by reductions in standard errors. To illustrate, consider the usual critical threshold of $t^*_{\alpha, \df-1}\approx 2$ and any observed t-statistic below $1$. Then, if $\Ry \leq .5$, it is \emph{impossible} to obtain a reversal that is not mainly driven by changes in point estimate. In other words, any post-mortem analysis of such significance reversals, using decompositions like (\ref{eq:t.tr}), must necessarily find that the relative change in bias is larger than the relative change in standard errors.  Overall, these results closely mirror  empirical patterns of reversals of statistical insignificance observed in applied research, such as those documented by \cite{lenz2021achieving}, and may offer a purely algebraic explanation for at least some of these patterns.

\begin{appendix}
\section*{Appendix: Deferred proofs}

\begin{proof}[Proof of Theorem~\ref{thm:t-max}]
    From (\ref{bf}) and (\ref{selambda}), the magnitude of the t-statistic for $H_0: \lambda = \lambda_0$ can be written as a function of $\Ry$ and $\Rd$,
    \begin{align}
        {\textstyle t(\Ry, \Rd)= 
        \frac{|(\hat{\lambda}_{r } - \lambda_0) \pm \BF \times \widehat{\se}(\hat{\lambda}_{r }) \times \sqrt{\df}|}{\widehat{\se}(\hat{\lambda}_{r })\times\SEF\times\sqrt{\frac{\df}{\df-1}}}.} \label{tref}
    \end{align}
    We wish to maximize $t(\Ry, \Rd)$  under the posited bounds $\Ry \leq \Rymax$ and  $\Rd \leq \Rdmax$. That is, we want to solve the constrained maximization problem,
    \begin{align}
        t_{\boldsymbol{R^2}}^{\max} &= \max_{\Ry, \Rd} t(\Ry, \Rd) \label{maxtref}
    \end{align}
    such that $\Ry \leq \Rymax$,  $\Rd \leq \Rdmax$. First notice that we should choose the direction of the bias that increases the magnitude of the difference $(\hat{\lambda} - \lambda_0)$. If $(\hat{\lambda}_{r } - \lambda_0) >0$ then we should add the $\BF$ term in (\ref{tref}), whereas if $(\hat{\lambda}_{r } - \lambda_0) < 0$ then the bias should be subtracted.
     Both cases yield the same (simplified) objective function as argued below.

    Suppose that $\hat{\lambda}_{r } > \lambda_0$. Then the absolute value from (\ref{maxtref}) can be dropped and the objective function becomes 
    \begin{align*}
       {\textstyle \frac{\hat{\lambda}_{r } - \lambda_0 + \BF \times \widehat{\se}(\hat{\lambda}_{r }) \times \sqrt{\df}}{\widehat{\se} (\hat{\lambda}_{r })\times\SEF\times\sqrt{\frac{\df}{\df-1}}}= 
        \frac{f_{r}  \times \widehat{\se}(\hat{\lambda}_{r }) + \BF \times \widehat{\se}(\hat{\lambda}_{r })}{\widehat{\se}(\hat{\lambda}_{r })\times\SEF\times\sqrt{\frac{1}{\df-1}}}. } 
    \end{align*}
    Now suppose that $\hat{\lambda}_{r } < \lambda_0$. We again have, 
    \begin{align*}
        {\textstyle  \frac{\lambda_0 - \hat{\lambda}_{r } + \BF \times \widehat{\se}(\hat{\lambda}_{r }) \times \sqrt{\df}}{\widehat{\se} (\hat{\lambda}_{r })\times\SEF\times\sqrt{\frac{\df}{\df-1}}} 
        = \frac{f_{r}  \times \widehat{\se}(\hat{\lambda}_{r }) + \BF \times \widehat{\se}(\hat{\lambda}_{r }) }{\widehat{\se}(\hat{\lambda}_{r })\times\SEF\times\sqrt{\frac{1}{\df-1}}}.}
    \end{align*}

    Therefore, after some algebraic manipulation, the maximum t-value in (\ref{maxtref}) will be of the form,
    \begin{align*}
        {\textstyle \tmax
        = \frac{f_{r} \sqrt{1-\Rdstar} + \sqrt{\Rystar\Rdstar}}{\sqrt{(1-\Rystar) / (\df-1)}}}
    \end{align*}
    where $\Rystar$, $\Rdstar$ are the values of $\Ry$ and $\Rd$ that optimize (\ref{maxtref}).
   
    We now find analytical expressions for the optimizers  $\Rystar$, $\Rdstar$.
    In what follows we write $t$  for the objective function with the understanding that it is written in its modified form above. The partial derivative of $\frac{t}{\sqrt{\df-1}}$ with respect to $\Ry$ is
    \begin{align*}
        {\textstyle \frac{\partial t / \sqrt{\df-1}}{\partial \Ry} = \tfrac{f_{r} \sqrt{(1-\Rd)\Ry}+\sqrt{\Rd}}{2(1-\Ry)^{\frac{3}{2}}\sqrt{\Ry}}.}
    \end{align*}
    Since the partial of $t$ with respect to $\Ry$ is always positive, we have $\Rystar = \Rymax$ unconditionally, i.e. $\Rystar$ always lies on the boundary.
    
    We now turn to the partial derivative of $\frac{t}{\sqrt{\df-1}}$ with respect to $\Rd$:
    \begin{align*}
        {\textstyle \frac{\partial t / \sqrt{\df-1}}{\partial \Rd} = \tfrac{-f_{r} \sqrt{\Rd}+\sqrt{\Ry(1-\Rd)}}{2\sqrt{(1-\Ry)(1-\Rd)\Rd}}.}
    \end{align*}
    It is straightforward to check that the second derivative of $\frac{t}{\sqrt{\df-1}}$ is negative with respect to $\Rd$. 
    Thus, when attainable, the zero of the first partial derivative with respect to $\Rd$ is a maximizer. Solving for the value that makes the first derivative zero yields: 
     \begin{align}
         \Rdstar = \frac{\Rymax}{f_{r} ^2 + \Rymax}. \label{xrvi1:d}
     \end{align}
    This interior point solution is only feasible when $f_{r} ^2 \geq \Rymax(1-\Rdmax)/\Rdmax$. Otherwise, if
    \begin{align}
         f_{r} ^2 < \Rymax(1-\Rdmax)/\Rdmax,
         \label{boundcondtn}
    \end{align}
    then the partial with respect to $\Rd$ is  strictly positive for all $\Rd \leq \Rdmax$ and so we obtain the boundary solution $\Rdstar = \Rdmax$.
\end{proof}

\begin{proof}[Proof of Theorem~\ref{thm:multi-tmax}]
Let $\mathbf{Z}$ denote an $(n~\times~m)$ matrix of  unobserved covariates and let $\boldsymbol{\hat{\gamma}}$ denote the coefficient vector of $\mathbf{Z}$. We are now interested in the long regression
\begin{align} 
    Y = \hat{\lambda}D + \mathbf{X}\hat{\beta} + \mathbf{Z}\hat{\gamma} + \hat{\epsilon}. \label{mv-regmodel}
\end{align}
Consider the $(n~\times~1)$ vector $Z_L := \mathbf{Z}\hat{\gamma}$. The regression 
\begin{align} 
    Y = \hat{\lambda}D + \mathbf{X}\hat{\beta} + Z_L + \hat{\epsilon} \label{zstar-regmodel}
\end{align}
yields the same value for $\hat{\lambda}$; therefore, the bias induced by $\mathbf{Z}$ is equal to that induced by $Z_L$  and $R^2_{Y \sim Z_L | D, \mathbf{X}} = R^2_{Y \sim \mathbf{Z} | D,\mathbf{X}}$. On the other hand, since $\hat{\gamma}$ is chosen solely to maximize $R^2_{Y \sim \mathbf{Z} | D,\mathbf{X}}$, we also have that $R^2_{D \sim Z_L | \mathbf{X}} \leq R^2_{D \sim \mathbf{Z} | \mathbf{X}}$. Now observe that the standard error formula from \eqref{selambda} holds for multivariate $\mathbf{Z}$ if we correctly adjust for the degrees of freedom \citep{cinelli_making_2020}. Further note that the bias induced by $Z_L$ is a strictly increasing function of $R^2_{D \sim Z_L | X}$. Thus, the most adversarial choice of $\mathbf{Z}$  is such that $R^2_{D \sim Z_L | \mathbf{X}} = R^2_{D \sim \mathbf{Z} | \mathbf{X}}$. We can thus assess the maximum t-statistic of a matrix $\mathbf{Z}$ by considering that of a single adversarial vector $Z_L$ and further adjusting for the degrees of freedom.
\end{proof}

\begin{proof}[Proof of Corollary~\ref{cor-phacking}] For any subset $\mathbf{X}_{S}$ of the columns of observed covariate matrix $\mathbf{X}$, recall that $R^2_{Y\sim\mathbf{X}_{S}|D} \leq \Ryx$ and $R^2_{D\sim\mathbf{X}_{S}} \leq \Rdx$. Now apply the proof of Theorem~\ref{thm:multi-tmax} with the alteration that the constraint $\Rdlz \leq \Rdmv$ is not necessarily tight, since $Z_L$ may not be adversarial for a specific dataset.
\end{proof}

For all cases below, consider the following condition for significance: 

\begin{align}
        \tstar \leq t^{\max}_{\boldsymbol{R^2}}. \label{maxineq}
\end{align}
\begin{proof}[Proof of Theorem~\ref{thm:XRVI-0}]
    First consider the case in which $f_{r} = 0$. This only happens if $\hat{\lambda}_r = \lambda_0$. Since here $\Rd=0$, the inclusion of $Z$ does not alter the point estimate and we still obtain $\hat{\lambda} = \lambda_0$ after adjusting for $Z$. Therefore, the adjusted t-statistic will be zero regardless of the value of the standard error.

    If $f_{r} > \fstar$ then we are already able to reject $H_0$ even if $Z$ has zero explanatory power.
    Otherwise, given the constraints $\Ry \leq \XRVI$ and $\Rd = 0$, the expression for $t_{\boldsymbol{R^2}}^{\max}$ simplifies to 
    \begin{align*}
    t_{\XRVI,0}^{\max} = \max_{\Ry} \frac{f_{r} }{\sqrt{1-\Ry}} \sqrt{\df-1} 
    \end{align*}
    such that $\Ry \leq \XRVI$.
    This is a strictly increasing function of $\Ry$ and thus attains its maximum at $\Rystar=\XRVI$. Thus solving for the minimum value of $\XRVI$ that satisfies (\ref{maxineq})
    is equivalent to solving for $\XRVI$ at the equality. That is,
    \begin{align*}
        \fstar &= \frac{f_{r} }{\sqrt{1-\XRVI^0_{\alpha}}}.
    \end{align*}
    Squaring and rearranging terms, we obtain
    \begin{align*}
        \XRVI^0_{\alpha} = 1-\left(\frac{f_{r} }{\fstar}\right)^2,
    \end{align*}
    as desired.
\end{proof}

\begin{proof}[Proof of Theorem~\ref{thm:XRVI-1}]
    If $f_{r} > \fstar$ then we are already able to reject $H_0$ even if $Z$ has zero explanatory power, and thus the minimal strength to reject $H_0$ is zero. 
    Otherwise, consider constraints $\Ry \leq \XRVI$ and $\Rd \leq 1$. From the proof of Theorem~\ref{thm:t-max} we see that $t_{\boldsymbol{R^2}}^{\max}$ is an increasing function of $\XRVI$. Thus solving for the minimum value of $\XRVI$ that satisfies (\ref{maxineq}) is equivalent to solving for $\XRVI$ at the equality. Notice that (\ref{boundcondtn}) is not satisfied here. We therefore plug in the interior point solution from Theorem~\ref{thm:t-max} for $t_{\boldsymbol{R^2}}^{\max}$ and solve for $\XRVI^{1}_{\alpha}$. This results in the equation
    \begin{align*}
        \fstar 
        &= \sqrt{\frac{f_{r} ^2 + \XRVI^{1}_{\alpha}}{1-\XRVI^{1}_{\alpha}}}
    \end{align*}
    which yields the solution 
    \begin{align*}
        \XRVI^{1}_{\alpha} = \dfrac{\fstarsq - f_{r} ^2}{1+\fstarsq},
    \end{align*}
as desired.
\end{proof}

\begin{theorem}[Closed-form expression for $\XRVI^{\Rdmax}_{\alpha}$]
\label{thm:xriv-r}
Let $\Rdmax > 0$ or $f_{r} > 0$. Then the analytical expression for $\XRVI^{\Rdmax}_{\alpha}$ is
\begin{align*}
    \XRVI^{\Rdmax}_{\alpha} =
    \begin{cases}
        0, &\text{if} \enspace q \leq 0,\\
        \XRVI^{1}_{\alpha}, &\text{if} \enspace {0 < q \leq \Rdmax}, \\
        \frac{-b + s\sqrt{b^2 - 4ac}}{2a}, &\text{otherwise},
    \end{cases}
\end{align*}
where
\begin{align}
    a &= 1 \label{a}, \\ 
    b &= {\scriptstyle - 2\left[ \frac{\fstarsq - (1-\Rdmax)f_r^2}{\fstarsq + \Rdmax} + \frac{2f_r^2(1-\Rdmax)\Rdmax}{(\fstarsq + \Rdmax)^2} \right]}, \label{b} \\
    c &= { \left[\frac{\fstarsq - (1-\Rdmax)f_r^2}{\fstarsq + \Rdmax}\right]^2,} \label{c} \\ 
    q &= { \frac{\fstarsq - f_r^2}{\fstarsq (1+ f_r^2)}},
\end{align}
and $s\in \{-1, 1\}$ is chosen to yield the valid quadratic root, i.e., $t_{\XRVI,\Rdmax}^{\max}=\tstar$. 
If $\Rdmax = 0$ and $f_{r} = 0$, there is no value of $\Ry$ capable of overturning an insignificant result.
\end{theorem}

\begin{proof}[Proof of Theorem~\ref{thm:xriv-r}]
As argued in the proofs of Theorems~\ref{thm:XRVI-0}~and~\ref{thm:XRVI-1}, if $\fstar \leq f_{r} $ (which implies $q\leq 0$), then the minimum strength of $Z$  necessary to attain significance is zero.  If $f_r = 0$ and $\Rdmax = 0$ then by Theorem~\ref{thm:XRVI-0}, we cannot overturn an insignificant result. Otherwise, solving for the minimum value of $\XRVI$ that satisfies (\ref{maxineq}) is equivalent to solving for $\XRVI$ at the equality. 

Now let $f_r>0$. Here we have two cases. Either both $R^2$ values reach the bound or the optimal value of $\Rd$ is an interior point. For the latter case, this means the constraint $\Rd$ is not binding and thus $\XRVI^{\Rdmax}_{\alpha}$ should equal  $\XRVI^{1}_{\alpha}$. Recall that $t$ is concave down with respect to $\Rdmax$; therefore, the optimal value of $\Rd$ is the interior point solution $\Rdstar$, as defined in \eqref{xrvi1:d}, if and only if $\Rdstar \leq \Rdmax$. We can simplify this condition to be of the form,
\begin{align*}
     q:=\frac{\fstarsq - f_r^2}{\fstarsq (1+ f_r^2)} \leq \Rdmax.
\end{align*}
It remains to solve for the case where both coordinates reach the bound. Here, the equality in (\ref{maxineq}) simplifies to
    \begin{align}
        {\textstyle \tstar = \frac{f_{r} \sqrt{1-\Rdmax} + \sqrt{\Rdmax\XRVI}}{\sqrt{(1-\XRVI)/(\df-1)}}.} \label{XRVI}
    \end{align}
    
We now solve for $\XRVI$ in (\ref{XRVI}) by squaring both sides and taking the valid root of the quadratic equation. The simplified quadratic form is
\begin{align}
        &{\scriptstyle \XRVI^2 - \left[ \frac{2(\fstarsq - (1-\Rdmax)f_r^2)}{\fstarsq + \Rdmax} + \frac{4f_r^2(1-\Rdmax)\Rdmax}{(\fstarsq + \Rdmax)^2} \right]\XRVI} \nonumber \\
        &{\scriptstyle \quad \quad + \left[\frac{\fstarsq - (1-\Rdmax)f_r^2}{\fstarsq + \Rdmax}\right]^2 = ~0. }
\end{align}
The expressions for $a,b,c$ in (\ref{a})-(\ref{c}) immediately follow from Sridharacharya-Bhaskara's formula for quadratic equations.  Finally, we note that if $f_r=0$ and $\Rdmax > 0$ then the $\XRVI^{\Rdmax}$ solution is valid. 
\end{proof}

\begin{proof}[Proof of Theorem~\ref{thm:RVI}] 
    The derivation of $\RVI_{\alpha}$ follows that of $\XRVI^{\Rdmax}_{\alpha}$ very closely (see proof of Theorem~\ref{thm:xriv-r}). The only difference is that when $f_{r} ^2 < 1-\RVI$, the equality in (\ref{maxineq}) simplifies to
    \begin{align*}
        \tstar = \frac{f_{r} \sqrt{1-\RVI} + \RVI}{\sqrt{1-\RVI}} \times \sqrt{\df-1}
    \end{align*}
    which is equivalent to 
    \begin{align}
        \fstar = f_{r}  + \frac{\RVI}{\sqrt{1-\RVI}}. \label{fstar}
    \end{align}
    Let $f_\Delta = \fstar - f_{r} $. Then (\ref{fstar}) further reduces to a quadratic function of $\RVI$ with positive root 
    \begin{align*}
        \RVI = \frac{1}{2}(\sqrt{f_\Delta^4 + 4f_\Delta^2}-f_\Delta^2).
    \end{align*} 
    It remains to show that (\ref{boundcondtn}),
       i.e. $f_{r} ^2 \geq 1-\RVI$,
    is equivalent to 
    \begin{align}
        f_{r}  > \frac{1}{\fstar}. \label{newcondtn}
    \end{align}
    This is equivalent to showing that $\RVI_{\alpha}$ is given by the interior point solution if and only if (\ref{newcondtn}) holds. To see why, recall that for the interior point solution, 
    \begin{align}
        \Rymax = \Rdmax = \RVI_{\alpha} = \dfrac{\fstarsq - f_{r} ^2}{1+\fstarsq}. \label{interiorsoln}
    \end{align}
    Plugging (\ref{interiorsoln}) into (\ref{boundcondtn}), we have
    \begin{align*}
        f_{r} ^2 &\geq 1-\RVI_{\alpha}  
        =  1-\frac{\fstarsq - f_{r} ^2}{1+\fstarsq}
    \end{align*}
    which reduces to
        $\fstar \geq \frac{1}{f_{r} }$.
\end{proof}

\end{appendix}
%
%

\begin{acks}[Acknowledgments]
The authors thank the referees, an Associate
Editor and the Editor for their helpful feedback. 
\end{acks}
\begin{funding}
This work is supported in part by the Royalty Research Fund at the University of Washington, the National Science Foundation under Grant No. MMS-2417955, and the Natural Sciences and Engineering Research Council of Canada under Grant No. 599253. 
\end{funding}

\begin{supplement}
\stitle{Online Supplement}
\sdescription{Section 1 provides a comparison of $\RVI_{\alpha}$ and $\XRVI_{\alpha}$ with $\RV_{\alpha}$, $\XRV_{\alpha}$ from \cite{cinelli_making_2020, cinelli2025omitted}. Section 2 derives the sampling distributions of $\RVI_{\alpha}$ and $\XRVI_{\alpha}$ both under the null and alternative hypothesis. Section~3 derives the sampling distribution of $\Rd$ when the treatment $D$ is randomized.}
\end{supplement}


\bibliographystyle{imsart-nameyear} 
\bibliography{bibliography}       


\end{document}